\newif\ifpictures
\picturestrue

\documentclass[12pt]{amsart}
\usepackage{amssymb,amsmath}
 \usepackage{amsopn}
 \usepackage{mathabx}
 \usepackage{xspace}
  \usepackage{hyperref}
 \usepackage[dvips]{graphicx}
\usepackage[arrow,matrix,curve]{xy}
\usepackage {color, tikz}

\numberwithin{equation}{section}
\newtheorem{thm}{Theorem}
\newtheorem{example}[thm]{Example}
\newtheorem{prop}[thm]{Proposition}

\newtheorem{cor}[thm]{Corollary}

\newtheorem{definition}[thm]{Definition}

\numberwithin{thm}{section}

\newcounter{FNC}[page]
\def\newfootnote#1{{\addtocounter{FNC}{2}$^\fnsymbol{FNC}$%
     \let\thefootnote\relax\footnotetext{$^\fnsymbol{FNC}$#1}}}

\newcommand{\N}{\mathbb{N}}

\newcommand{\R}{\mathbb{R}}

\newcommand{\Z}{\mathbb{Z}}
\newcommand{\ulc}{\underline{c}}
\newcommand{\ult}{\underline{t}}
\newcommand{\uls}{\underline{s}}
\newcommand{\ulu}{\underline{u}}
\newcommand{\ule}{\underline{e}}
\newcommand{\ulzero}{\underline{0}}

\newcommand{\alp}{\alpha}

\newcommand{\NN}{\mathbb{N}}
\newcommand{\RR}{\mathbb{R}}
\newcommand{\ZZ}{\mathbb{Z}}

\newcommand{\sbs}{\subseteq}
\newcommand{\sps}{\supseteq}

\DeclareMathOperator{\conv}{conv}

\DeclareMathOperator{\New}{New}

\DeclareMathOperator{\Int}{int}

\title[Lower Bounds for Polynomials based on Geometric Programming]{Lower Bounds for Polynomials with Simplex Newton Polytopes Based on Geometric Programming}

\author{Sadik Iliman} \author{Timo de Wolff}

\address{Sadik Iliman, Goethe-Universit\"at, FB 12 -- Institut f\"ur Mathematik,
Postfach 11 19 32, 60054 Frankfurt am Main, Germany\medskip \newline 
\hspace*{8pt} Timo de Wolff, Texas A\&M University, Department of Mathematics, College Station, TX 77843-3368, 
 USA\medskip}

\email{iliman@math.uni-frankfurt.de \\
dewolff@math.tamu.edu }


\begin{document}

\begin{abstract}
In this article, we propose a geometric programming method in order to compute lower bounds for real polynomials. 
We provide new sufficient conditions for polynomials to be nonnegative as well as to have a sum of binomial squares representation. These criteria rely on the coefficients and the support of a polynomial and generalize all previous ones by Lasserre, Ghasemi, Marshall, Fidalgo and Kovacec to polynomials with arbitrary simplex Newton polytopes.

This generalization yields a geometric programming approach for computing lower bounds for polynomials that significantly extends the geometric programming method proposed by Ghasemi and Marshall. Furthermore, it shows that geometric programming is strongly related to nonnegativity certificates based on sums of nonnegative circuit polynomials, which were recently introduced by the authors.
\end{abstract}

\maketitle

\section{Introduction}
Finding lower bounds for real polynomials is a central problem in polynomial optimization. For polynomials with few variables or low degree, and for polynomials with additional structural properties there exist several well working approaches to this problem. The best known lower bounds are provided by \emph{Lasserre relaxations} using semidefinite programming. Although the optimal value of a semidefinite program can be computed in polynomial time (up to an additive error), the size of such programs grows rapidly with the number of variables or degree of the polynomials. Hence, there is much recent interest in finding lower bounds for polynomials using alternative approaches such as \emph{geometric programming} (see \eqref{Equ:GP} for a formal definition). Geometric programs can be solved in polynomial time using interior point methods \cite{nesterov}; see also \cite[Page 118]{Boyd:GP}. In this article, we provide new lower bounds for polynomials using geometric programs. These bounds extend results in \cite{GP} by Ghasemi and Marshall.\\

Let $\R[x] = \R[x_1,\ldots,x_n]$ and let $\R[x]_{d} = \R[x_1,\dots,x_n]_{d}$ be the space of polynomials of degree $d \in \N$. A \textit{global polynomial optimization problem} for some $f \in \R[x]_{2d}$ is given by
\begin{eqnarray*}
f^* & = &\inf\{f(x) : x\in\R^n\} \ = \ \sup\{\lambda\in\R : f - \lambda \geq 0\}.
\end{eqnarray*}

It is well-known that in general computing $f^*$ is NP-hard \cite{Dickinson:Gijben}. By relaxing the nonnegativity condition to a sum of squares condition, a lower bound for $f^*$ based on semidefinite programming is given by
\begin{eqnarray*}
	f_{\rm sos} & = & \sup\left\{\lambda \in \mathbb R \ : \ f - \lambda = \sum_{i=1}^k q_i^2 \ \text{ for some } \ q_i\in\R[x]\right\}
\end{eqnarray*}
and hence $f_{\rm sos} \leq f^*$ \cite{Lasserre:Buch}. A central open problem in polynomial optimization is to analyze the gap $f^* - f_{\rm sos}$. Very little is known about this gap beyond the cases where it always vanishes. This happens exactly for $(n,2d) \in \{(1,2d),(n,2),(2,4)\}$ by Hilbert's Theorem \cite{Hilbert:Seminal}. For an overview about the topic see \cite{Blekherman:Parrilo:Thomas, Lasserre:Buch, Laurent:Survey}.

Let $e_1,\ldots,e_n$ denote the standard basis of $\R^n$. In \cite{Fidalgo:Kovacec}, Fidalgo and Kovacec consider the class of polynomials, whose Newton polytope is a scaling of the standard simplex $\conv\{0,2d\,e_1,\dots,2d\,e_n\}$. For these polynomials they provide certificates, i.e., sufficient conditions, both for nonnegativity and for being a sum of squares. In \cite{GP} Ghasemi and Marshall show that these certificates can be translated into checking feasibility of a geometric program. Moreover, in their recent works \cite{GP,GP:semialgebraic} Ghasemi and Marshall show several important further facts for polynomial optimization via geometric programming. Two key observations are the following ones:
\begin{enumerate}
	\item For general polynomials lower bounds based on geometric programming are seemingly not as good as bounds obtained by semidefinite programming.
	\item Even higher dimensional examples can often be solved quite fast via geometric programming. In contrast, semidefinite programs often do not provide an output for problems involving polynomials with many variables or of high degree (at least with the current SDP solvers).
\end{enumerate}

\medskip

The extension of Ghasemi and Marshall's results, which we provide in this article, relies on the following key observation. In addition to the sum of squares approach, one can use \textit{nonnegative circuit polynomials} to certify nonnegativity. Nonnegative circuit polynomials  were recently introduced by the authors in \cite{Iliman:deWolff:Circuits}. Particularly, the authors results in \cite{Iliman:deWolff:Circuits} imply as a special case the sufficient condition for nonnegativity by Fidalgo and Kovacec \cite{Fidalgo:Kovacec}, which was used by Ghasemi and Marshall. Therefore, it is self-evident to ask whether the translation into geometric programs can also be generalized. The purpose of this article is to show that this is indeed the case.

Let $f \in \R[x]$ be a real polynomial with simplex Newton polytope such that all its vertices are in $(2\N)^n$ and the coefficients of the terms corresponding to the vertices are nonnegative. The main theoretical results we contribute in Section \ref{Sec:MainResults} are some easily checkable criteria on the coefficients of such a polynomial $f$, which imply that $f$ is nonnegative. More precisely, these criteria imply that $f$ is a \textit{sum of nonnegative circuit polynomials ({\sc sonc})}, and, as a consequence of results by the authors in \cite{Iliman:deWolff:Circuits}, every {\sc sonc} is nonnegative. See Theorems \ref{thm:main1} and \ref{thm:main2}, and see Section \ref{Sec:Preliminaries} for a formal definition of a {\sc sonc}. Moreover, we provide a second criterion on the support of a {\sc sonc} polynomial, which implies that the {\sc sonc} additionally is a sum of binomial squares.

The key observation is that, as in \cite{GP}, these criteria can be translated into a geometric optimization problem (Corollary \ref{cor:gp}) in order to find a lower bound for a polynomial. As a surprising fact we show in Corollary \ref{cor:gpbessersonc} that for very rich classes of polynomials with simplex Newton polytope, the optimal value $f_{\rm gp}$ of the corresponding geometric program is at least as good as the bound $f_{\rm sos}$. In fact $f_{\rm gp}$ \textit{is} $f^*$ in these cases. This is in sharp contrast to the general observation by Ghasemi and Marshall \cite{GP,GP:semialgebraic}, which we outlined above in (1). Additionally, based on similar examples, we can see that the computation of $f_{\rm gp}$ is much faster than in the corresponding semidefinite optimization problem, as it was already shown numerically in \cite{GP}. 

Using the geometric programming software package \textsc{gpposy} for \textsc{Matlab}, we demonstrate the capabilities of our results on the basis of different examples. A major observation is that the bounds $f_{\rm gp}$ and $f_{\rm sos}$ are not comparable in general, since the convex cones of {\sc sonc}'s and sums of squares do not contain each other, see \cite{Iliman:deWolff:Circuits}. This observation, again, is in sharp contrast to the one in \cite{GP} where the bound $f_{\rm sos}$ and the bound given by Ghasemi and Marshall's geometric program are comparable.

Furthermore, we show in Section \ref{Sec:Constrained} that our methods are not only applicable to global polynomial optimization problems, but also to constrained ones using similar methods as Ghasemi and Marshall in \cite{GP:semialgebraic}. A further discussion of the constrained case and generalizations of the results in Section \ref{Sec:Constrained} and in \cite{GP:semialgebraic} is content of the follow-up article \cite{Dressler:Iliman:deWolff:GPConstrained}.

\section{Preliminaries}
\label{Sec:Preliminaries}

We consider a polynomial $f \in \R[x]_{2d}$ of the form $f = \sum_{\alpha \in \N^n}^{} f_{\alpha}x^{\alpha}$ with $f_{\alpha} \in \R$, $x^{\alpha} = x_1^{\alpha_1} \cdots x_n^{\alpha_n}$, and Newton polytope $\New(f) = \conv\{\alpha \in \N^n : f_{\alpha} \neq 0\}$. We call a lattice point \textit{even} if it is in $(2\N)^n$. A {\it binomial} is 
an expression of the form $rx^\alpha + sx^\beta$  with $r,s\in \RR;$ if 
at least one of $r,s$ is 0, it is (also) a monomial.  If a polynomial is a sum of squares of binomials, then it is customary to abbreviate
this by saying it is a {\sc sobs}, meaning it is a ``sum of binomial squares''.
We denote by $\conv(S)$  the convex hull   of a subset $S$ of $\RR^n.$
Our interest lies foremost in \textit{ST-polynomials} which we define as follows.

\begin{definition}
An \emph{ST-polynomial written in standard form} is a polynomial of the form
\begin{eqnarray*}
 f & = & f_{\alpha(0)}+\sum_{j=1}^n f_{\alpha(j)} x^{\alpha(j)} +\sum_{\alpha \in \Delta} f_\alpha x^\alpha,
\end{eqnarray*}
with  exponents $\alpha(j)$ and $\alpha$, coefficients $f_{\alpha(j)}, f_\alpha$, and a set $\Delta$ 
for which the following hold:

\begin{description}
 \item[(ST1)] The points $\alpha(0)=0$ and $ \alpha(1), \alpha(2),\ldots,\alpha(n)$ define a set $V$ of affinely independent, even points in $(2\NN)^n.$
 \item[(ST2)] $\Delta$ is the set of exponents $\alpha$ in $f$ not defining \textit{monomial squares};
 i.e. $\alpha\in \Delta$ iff $f_\alpha<0$ or $\alpha\not \in (2\NN)^n$ and $f_{\alp} > 0$.
 \item[(ST3)] There holds the inclusion $\Delta \sbs \conv(V);$ or, equivalently,  
 every $\alpha \in \Delta$ can be written uniquely as 
 \begin{eqnarray*}
  & & \alpha \ = \ \sum_{j=0}^n \lambda_j^{(\alpha)} \alpha(j) \ \text{ with } \ \lambda_j^{(\alpha)} \ \geq \ 0 \ \text{ and } \  \sum_{j=0}^n \lambda_j^{(\alpha)} \ = \ 1.
 \end{eqnarray*}
 \item[(ST4)] $f_0=f_{\alpha(0)}, f_{\alpha(1)}, \cdots, f_{\alpha(n)}$ are nonnegative and if 
     $f_{\alpha(j)}=0,$ then for all $\alpha \in \Delta$ there holds $\lambda_j^{(\alpha)}=0$.
\end{description}
The ``ST'' in ``ST-polynomial'' is short for ``simplex tail''. The tail part is
given by the sum $\sum_{\alpha\in \Delta} f_{\alpha} x^{\alpha},$ while the other terms define the simplex part. 
\end{definition}

Note that hypothesis (ST1) implies that  $V=\{\alpha(0),\ldots,\alpha(n)\}$ is the vertex set of  
a full dimensional simplex. It consists of even lattice points and has one vertex at the origin. 
Hypothesis (ST3) implies that the 
Newton polytope of $f$ is a face of  this simplex, possibly the simplex itself. 
 We will have $\New(f)=\conv(V)$ if and only if 
all $f_{\alpha(j)}$ are positive.  Otherwise, if $f_{\alpha(j)}=0$ for some $j$, then $\New(f)$ is a proper face 
of $\conv(V).$ Hypotheses (ST2), (ST3) and (ST4) together imply that $\Delta\sbs \New(f)\setminus V.$
So, $\Delta$ is uniquely defined by $f$ and may be referred to by  $\Delta(f).$

The $\lambda_j^{(\alpha)}$ denote the barycentric coordinates of $\alpha$ relative to the vertices $\alpha(j),$  
$j=0,\ldots,n.$   An ST-polynomial is homogeneous only if $f_0=0$
and for all  $\alpha, \alpha(i), \alpha(j)$ occurring in $f$ (associated to nonzero coefficients) we have 
$|\alpha(i)| = |\alpha(j)| = |\alpha|,$ where  $|\alpha|=\sum_{i=1}^n \alpha_i.$
In this case $\lambda_0^{(\alpha)}=0$ holds for all $\alpha\in \Delta$; the converse needs not to be true.\\
Given a polynomial $f,$  well established algorithms from convex geometry allow to 
determine if $f$ is ST and if so to rewrite it in this form. \\

\begin{definition}
An ST-polynomial is  a \emph{circuit polynomial} if $\Delta(f)$ is empty or a singleton.
 We fix the standard
notation for a circuit polynomial $f$ and define the associated \emph{circuit number} $\Theta_f$, which was first defined by the authors in \cite{Iliman:deWolff:Circuits}, as follows:
\begin{eqnarray}
 & & f \ = \ f_0+\sum_{j=1}^n f_{\alpha (j)}x^{\alpha(j)} +c x^\alpha, \ \text{ and } \ \Theta_f \ = \ \prod_{j\in {\rm nz}(\alpha)} \left(\frac{f_{\alpha(j)}}{\lambda_j^{(\alpha)}}\right)^{\lambda_j^{(\alpha)}} \label{Equ:DefCircuitNumber}
\end{eqnarray}
with ${\rm nz}(\alpha)=\{j\in \{0,\ldots,n\}:\lambda_j^{(\alpha)}\neq 0\}.$ In the uninteresting case $c=0,$ i.e. $\Delta(f)=\emptyset,$ define $\Theta_f=1$.
\end{definition}
Much of our work will be centered around writing an ST-polynomial as a sum of nonnegative circuit polynomials ({\sc sonc}).

Let $e_i=(\delta_{i1},\ldots,\delta_{in})$ be the $i$-th standard vector.
The class of  ST-polynomials  covers in essence the class of polynomials of 
degree $2d$ considered by Ghasemi and Marshall, see, e.g., \cite[Theorem 2.3 and Corollary 2.5]{GP}. This can be seen by putting 
$\alpha(j)= 2de_j$  for $j=1,\ldots,n$  and noting that an $\alpha$ with $|\alpha|\leq 2d$ lies in the convex hull of $\alpha(0)=0,2de_1,\ldots,
2de_n.$ Note that Ghasemi and Marshall admit in the definition of their polynomials $f$
larger sets $\Omega$ of exponents than $\Delta$. However, the difference 
$\Omega \setminus \Delta$ is associated to terms $f_\alpha x^\alpha$ that are monomial squares. All their criteria for the sum of squares 
property and algorithms for lower bounds of polynomials use virtually without exception only the information on the coefficients 
$f_{\alpha(j)}$ and $f_\alpha$ with $\alpha\in \Delta.$  The same will hold in this paper. 
Everything we could say on basis of the present investigation  for more general classes of polynomials would be a 
trivial consequence of what we find here for ST-polynomials. This is the reason why we concentrate on the class of polynomials 
defined above. \\

A fundamental fact is that nonnegativity of a circuit polynomial $f$ can be completely decided by comparing its tail coefficient
with its circuit number $\Theta_f.$ 

\begin{thm}[\cite{Iliman:deWolff:Circuits}, Theorem 3.8]
\label{thm:positiv}
 Let $f$  be a  circuit polynomial in standard form and $\Theta_f$ its circuit number, as defined in \eqref{Equ:DefCircuitNumber}. 
Then the following are equivalent:
\begin{enumerate}
 \item $f$ is nonnegative.
 \item $|c| \leq \Theta_f$ and $\alpha\not \in (2\NN)^n$ \quad or \quad $c\geq -\Theta_f$ and $\alpha \in (2\NN)^n$ \quad or \quad $\Delta(f) = \emptyset.$
\end{enumerate}
\end{thm}

Note that (2) can be equivalently stated as: $|c| \leq \Theta_f$ or $c = 0$ or $f$ is a sum of monomial squares. 
At this point we remark that the definition of the circuit number $\Theta_f$ slightly differs from the one used in \cite{Iliman:deWolff:Circuits}. The reason is that $\alpha\in\Delta(f)$ is not necessarily assumed to be an interior point as it is in \cite{Iliman:deWolff:Circuits}. The difference between the two definitions is explained by \cite[Lemma 3.7]{Iliman:deWolff:Circuits}.

Writing a polynomial as a sum of nonnegative circuit polynomials is a certificate of nonnegativity. Let us denote
by  {\sc sonc}  the class of polynomials that are {\it sums of nonnegative circuit polynomials} or the property of a polynomial
to be in this class.  We show that this class brings new insights to polynomial optimization and hence deserves a place among the well established classes {\sc sos} (sums of squares) and {\sc sobs} (sums of binomial squares). For further details about \textsc{sonc}'s see \cite{deWolff:Circuits:OWR,Iliman:deWolff:Circuits}

\begin{example}
 Investigate the family $m_c(x,y)=1+x^2y^4+x^4y^2+c x^2y^2$. There needs to exist a smallest negative $c \in \R_{< 0}$ such that $m_c(x,y)$ is nonnegative. For negative $c$ the corresponding polynomial $m$ will be a circuit polynomial since $f_{(0,0)}=f_{(2,4)}=f_{(4,2)}=1$ and
$(2,2)=\frac{1}{3}(0,0)+\frac{1}{3}(2,4)+\frac{1}{3}(4,2).$ We obtain the circuit number 
\begin{eqnarray*}
 \Theta_{m_c} & = & \left(\frac{1}{1/3}\right)^{1/3} \cdot \left(\frac{1}{1/3}\right)^{1/3} \cdot \left(\frac{1}{1/3}\right)^{1/3} \ = \ 3.
\end{eqnarray*}
Thus, $m_c$ is nonnegative for $c\geq -3$ but no smaller real number. 
The polynomial obtained for $c = - 3$ is the well-known Motzkin polynomial which is known not to be a sum of squares.
\end{example}

\begin{example}
In \cite[Theorem 2.3]{Fidalgo:Kovacec} forms (homogeneous polynomials) denoted by
  $E(x)=b_1x_1^{2d}+\cdots +b_n x_n^{2d}-\mu  x_1^{a_1}\cdots x_n^{a_n}$ with $b_i\geq 0$ are analyzed concerning nonnegativity. $E(x)$ is an instance of a circuit polynomial (in \cite{Fidalgo:Kovacec} called elementary 
diagonal minus tail form) with $c=-\mu$ in the sense of Theorem \ref{thm:positiv}. Since $\alpha=(a_1,\ldots,a_n)=\sum_{i=1}^n \frac{a_i}{2d} 2d e_i,$ we obtain
\begin{eqnarray*}
 \Theta_E & = & \prod_{i\in \rm nz (\alpha)}\left(\frac{b_i}{(a_i/2d)}  \right)^{a_i/2d}
    \ = \ 2d \prod_{\substack{i=1 \\ a_i\neq 0}}^n \left(\frac{b_i}{a_i}  \right)^{a_i/2d},
\end{eqnarray*}
as the threshold value for nonnegativity.
\end{example}

In the earlier paper \cite{Iliman:deWolff:Circuits}, a criterion for a polynomial with simplex Newton polytope to be a {\sc sonc} was established.

\begin{thm}{\cite[Corollary 7.4]{Iliman:deWolff:Circuits}}
\label{thm:SOSMultiple}
Let $f$  be a nonnegative ST-polynomial in standard notation and $\Delta(f) \subseteq \Int(\New(f)\cap \mathbb N^n)$. If there exists a point
 $v\in (\R^*)^n$ such that $f_\alpha v^\alpha <0$ for all $\alpha\in \Delta(f)$, then $f$ is a {\sc sonc} and the circuit polynomials entering in this {\sc sonc} decomposition will all have the same Newton polytope as $f$. 
 \end{thm}

The relation between \textsc{sonc}'s and \textsc{sobs} is clarified by using definitions and results from Reznick \cite{Reznick:AGI}. Note that we find it convenient to remain near to Reznick's own notation in this discussion till the end of the
proof of Proposition 2.7. He
defines, given a set $L\sbs \Z^n,$ the sets of averages 
$$A(L) \ = \ \left\{\frac{1}{2} (s + t): s, t \in L \cap (2\ZZ)^n\right\}\,\, \textrm{and}\,\, \bar{A}(L) \ = \ \left\{\frac{1}{2} (s + t): s, t \in L \cap (2\ZZ)^n, s\neq t\right\}.$$
Given a set $P \sbs (2\ZZ)^n$, $L$ is $P$-{\it mediated} if $P \sbs L \sbs \bar{A}(L) \cup P;$
see \cite[p. 433,438]{Reznick:AGI}.
He then shows that there exists, given $P,$ a maximal $P$-mediated set $P^*$ that contains 
{\it every} $P$-mediated set. Reznick explains things in the context of ``frameworks'', that are sets of even lattice points 
that all have the same 1-norm. But his algorithmic construction of $P^*$ works literally for any finite set of even lattice points. 
To find $P^*$ one begins with $P^0=C(P)= \conv(P) \cap \ZZ^n$ and constructs via
$P^{k+1}=\bar A(P^k)\cup P$ inductively a contracting sequence of 
sets $P^0\sps P^1 \sps P^2 \sps\ldots$ \, . As $P^0$ is finite, this sequence  becomes stationary at a set 
which is shown to be the required $P^*.$ If the convex hull of $P$ is a simplex and $P^* =C(P),$ then Reznick calls $P$ an 
$H$\textit{-trellis} \cite[p436c1]{Reznick:AGI}. The ``$H$'' is borrowed from the fact that the \textit{Hurwitz form} has the standard simplex as its Newton polytope. The standard simplex has the property that its corresponding $P^*$ satisfies $P^* = C(P) = \conv(P) \cap \Z^n$. We shall speak of an $H$\textit{-simplex}. \\

We will apply Reznick's construction to the vertex set $V=V(f)$ of the Newton polytope of an ST-polynomial $f$ and write, with slight abuse of notation, $\New(f)^*$ for $V^*.$ As mentioned, the papers \cite{Fidalgo:Kovacec,GP} deal exclusively with polynomials $f$ in which $V(f)$ consists
of the scaled standard vectors and, possibly, $0=\alpha(0).$ To put the results of those papers into perspective, 
we show that 
the inhomogeneous simplices generated by the origin and the standard vectors, just as their homogeneous 
counterparts are also $H$-simplices.  

\begin{prop}\label{prop:trellis}
$\mathfrak{H}=\{0,2d\,e_1,\ldots,2d\,e_n\}$ is a nonhomogeneous $H$-simplex; that is, 
$\mathfrak{H}^*=C(\mathfrak H).$ Every subset of $\mathfrak{H}$ again defines an $H$-simplex.
\end{prop}

\begin{proof}
 We adapt Reznick's proof who shows an analogous fact for the case $\{2de_1,\ldots,2de_n\}.$ 
It is clear that $\mathfrak{H}$ is a nonhomogeneous trellis in the sense that the vertices define a  simplex which in 
our case is full dimensional. By Reznick's criterion \cite[p438c-6]{Reznick:AGI} and notation $E(.)=C(.)\cap (2\ZZ)^n,$ we have to show that 
$\bar{A}(E(\mathfrak H))=C(\mathfrak H)\setminus 
\mathfrak H$. (The reader should note that Reznick's criterion is actually false in the context said there - frameworks - but it is
true for trellises, in particular $\mathfrak H$.) The points in $\mathfrak H$ have at most one nonzero entry. If this happens for a point $\ulu$ in 
$\bar{A}(E(\mathfrak H))$, then  the two distinct points in $E(\mathfrak H)$ with average $\ulu$ must also have at most one 
nonzero entry at the same position as $\ulu$. But then distinctness implies $|\ulu|<2d$ and so $\ulu\not \in \mathfrak H.$ So 
$\bar{A}(E(\mathfrak H)) \sbs C(\mathfrak H)\setminus \mathfrak H.$  Now consider a point 
$\ulc \in C(\mathfrak H)\setminus \mathfrak H.$

Case 1: $\ulc$ has only one nonzero coordinate. Then $\ulc=c \cdot \ule_i$ for some integer $c$ with $0<c <2d.$ Points
in $E(\mathfrak H)$ whose average might yield $\ulc$ are necessarily points of the form $0\ule_i, 2\ule_i, 4\ule_i,\ldots, 2d \ule_i.$  If $c$ is odd then $c - 1$ and $c + 1$ are even integers in $\{0,\ldots,2d\}$ and $\ulc= \frac{1}{2} ((c-1)\ule_i + (c+1)\ule_i)$ is the desired representation. If $c$ is even, then we may conclude analogously using $c - 2,c + 2$ instead. 

Case 2: $\ulc=(c_1,\ldots,c_n)$ has two or more nonzero coordinates. Then we may assume by symmetry and for ease of future notation, 
that  $1\leq c_1\leq c_n .$ Since $\ulc \in C(\mathfrak H) \setminus \mathfrak{H} \sbs \ZZ_{\geq 0}^n,$ we have $c=|\ulc|\in \{1,\ldots,2d\}$ and 
$1\leq c_1 \leq c/2.$ 
We define $b_r=\sum_{i=1}^r c_i$ and find the unique  $k\in \{1,\ldots,n-1\}$ such that 
$b_k \leq c/2 \leq b_{k+1}.$ Define $\dot c = c + 1$ if $c$ is odd and  $\dot c = c$ if $c$ is even.

Now we let $\uls = (2c_1,\ldots,2c_k,\dot c-2b_k,0,\ldots,0)$ and $\ult=(0,\ldots,0,2b_{k+1}-\dot c,2c_{k+2},\ldots,2c_n).$
Assume $2b_{k+1}-c = 0.$ Then $\dot c = c$ and hence also $k+1 < n,$ since otherwise $2b_n-c = c=0.$  Thus, $\uls, \ult\neq 0.$
If $c$ is odd, then $2b_{k+1} > c > 2b_k$.

Therefore, in all cases $\uls,\ult> 0,$ and $\uls+\ult = 2\ulc.$  So, $\uls,\ult $ are even lattice points whose
average is $\ulc.$ Finally note that we have $\uls=\sum_{i=1}^k \frac{2c_i}{2d} 2d\ule_i + \frac{\dot c-2b_k}{2d} (2d e_{k+1})+ \frac{2d-\dot c}{2d} \ulzero$ as a convex combination of points in $\mathfrak H$
yielding  $\uls$. Thus, showing $\uls\in E(\mathfrak H).$ Similarly, one shows that $\ult \in E(\mathfrak H).$
Thus $\bar{A}(E(\mathfrak H)) \sps C(\mathfrak H)\setminus \mathfrak H$ and we have proven the first part. 

It is clear by deleting zeros in the coordinates which are not equal to $i_1,\ldots,i_k$ that subsets of the form $\{0,e_{i_1},\ldots,e_{i_k}\},$ $k\leq n$ will again define $H$-simplices. From this and
Reznick's own result the claim concerning subsets follows.
\end{proof}

\begin{example}
Figure \ref{Fig:Simplices} shows a scaled standard simplex for the case $d=3, n=2.$
\end{example}

\begin{figure}
\ifpictures
$\includegraphics[width=0.35\linewidth]{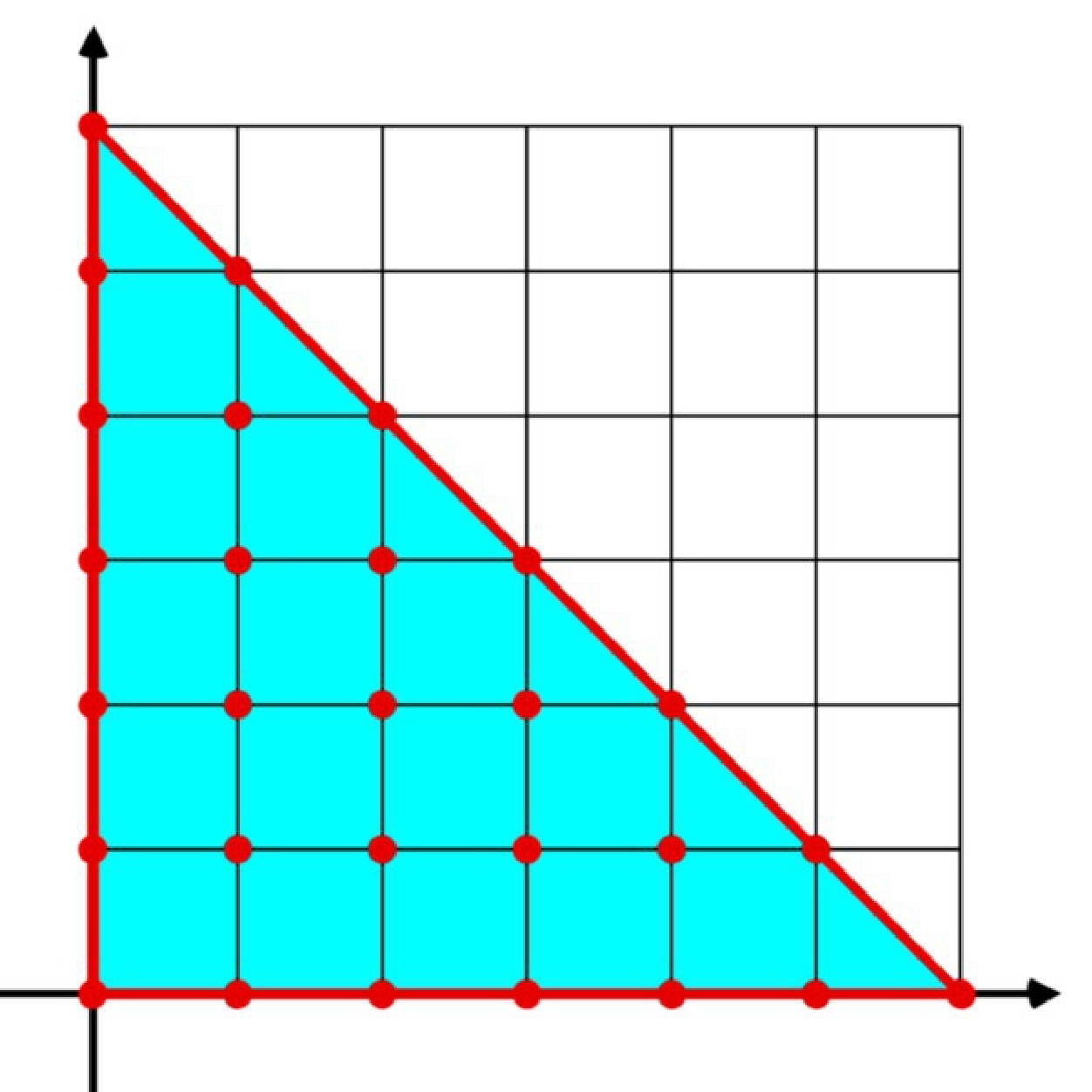}$
\fi
\caption{The $H$-simplex $\conv\{(0,0),(6,0),(0,6)\} \subset \R^2$. All lattice points are contained in the corresponding maximal mediated set.
}
\label{Fig:Simplices}
\end{figure}

Using \cite[Corollary 4.9]{Reznick:AGI}, which gives a necessary and sufficient criterion when ``simplicial agiforms'' are
\textsc{sobs}, we developed in \cite{Iliman:deWolff:Circuits} the following theorem. 

\begin{thm}[\cite{Iliman:deWolff:Circuits}, Theorem 5.2]\label{thm:sos}
Let $f$ be a nonnegative circuit polynomial in standard notation, as defined in \eqref{Equ:DefCircuitNumber}. Then the following statements are equivalent:
\begin{enumerate}
 \item $f$ is a sum of squares,
 \item $f$ is a sum of binomial squares,
 \item $\alpha \in  \New(f)^*$.
\end{enumerate}
\end{thm}

From the Theorems \ref{thm:SOSMultiple} and \ref{thm:sos} we obtain the following corollary.

\begin{cor}[\cite{Iliman:deWolff:Circuits}, Corollary 7.4]
\label{thm:SBSMultiple}
 Let $f$ be an ST-polynomial in standard form such that $\Delta(f) \subseteq \Int(\New(f)\cap \mathbb N^n)$. If there exists $v\in (\R^*)^n$ with $f_\alpha v^\alpha <0$ for all $\alpha\in \Delta(f)$ and  
  $\Delta(f) \sbs \New(f)^*$, then $f$ is nonnegative if and only if $f$ is a sum of binomial squares. 
\end{cor}

Thus, by Theorem \ref{thm:sos}, a circuit polynomial with $H$-simplex Newton polytope is nonnegative if and only if it is
a sum of squares. For the very special case of the scaled standard simplex and expressed in the homogeneous case this is the main result of \cite{Fidalgo:Kovacec}.

In \cite[Theorem 5.9]{Iliman:deWolff:Circuits}, sufficient conditions based on 2-normality of polytopes and toric
geometry are given for a simplex to be an $H$-simplex. In particular, every sufficiently large simplex with even vertices is an $H$-simplex in $\R^n$; see \cite[Section 5.1]{Iliman:deWolff:Circuits} for details.

\begin{example}
 Let $f = \frac{7}{12} + x_1^6 + x_2^4 + c · x_1x_2$ with $c \in  \R^*$. Note that the interior
lattice point $(1, 1)^T$ has the barycentric coordinates $(\lambda_0, \lambda_1, \lambda_2) = ( \frac{7}{12},\frac{1}{6},\frac{1}{4})$ in terms of the vertices of $\New(f)$.  
By Theorem \ref{thm:positiv} $f$ is nonnegative if and only if $|c| \leq 6^{\frac{1}{6}} 4^{\frac{1}{4}}.$
Once it is determined that $f$ is nonnegative, the question whether or not $f$ is a sum of
squares depends solely on the lattice point configuration of New$(f).$ It is easy to check
that $\New(f)^* = \New(f) \cap \NN^2.$ Hence, in particular the ``inner'' term $c \cdot x_1x_2$ has an
exponent $(1, 1)^T \in \New(f)^*$. Therefore, by Theorem \ref{thm:sos}, $f$ is a sum of binomial squares
if $f$ is nonnegative.
\end{example}

As mentioned before, the results of the preliminaries are, up to slight variations, taken from the article \cite{Iliman:deWolff:Circuits}. For further background on the results used here the interested reader should particularly focus on Section 3 for nonnegativity of circuit polynomials, Section 5 for the relation to \textsc{sos}, and on Section 7 for the structure of the \textsc{sonc} cone in \cite{Iliman:deWolff:Circuits}. Moreover, a short overview about these results can be found in the Oberwolfach report \cite{deWolff:Circuits:OWR}.

\section{Main Results}
\label{Sec:MainResults}
In this section, we provide sufficient criteria on the coefficients of a polynomial $f$ which
imply that $f$ is a sum of nonnegative circuit polynomials or a sum of (binomial)
squares and therefore that $f$ is nonnegative. We introduce a new lower bound for nonnegativity,
which we will later relate to geometric programs. 

\begin{thm}
\label{thm:main1}
Let $f$ be an ST-polynomial in standard form. Assume that for every pair 
$(\alpha,j) \in \Delta(f)\times \{0,1,\ldots,n\} $ 
there exists an $a_{\alpha,j} \geq 0$  such that the following holds:
\begin{enumerate}
 \item $|f_\alpha| \leq \prod\limits_{j\in {\rm nz(\alpha)}} \left(\frac{a_{\alpha,j}}{\lambda_j^{(\alpha)}}\right)^{\lambda_j^{(\alpha)}}$,
 \item $f_{\alpha(j)}\geq \sum\limits_{\alpha\in \Delta(f)} a_{\alpha,j},$ \quad for all $j = 0,\ldots,n.$
\end{enumerate}
Then $f$  is a sum of $|\Delta(f)|$ nonnegative circuit polynomials ({\sc sonc}), which all have faces of $\New(f)$ as Newton polytopes.
If in addition $\Delta(f) \sbs \New(f)^*,$ then $f$ is a sum of binomial squares. 
\end{thm}

\begin{proof}
 Fix an $\alpha\in \Delta(f).$  Then $|f_\alpha|>0$ and  the condition (1) guarantees that if $\lambda_j^{(\alpha)}>0$, then $a_{\alpha,j} >0.$ Together with the nonnegativity of the $a_{\alpha,j}$ this guarantees that $g_\alpha= a_{\alpha,0}+\sum_{j=1}^n a_{\alpha,j}x^{\alpha(j)} +f_\alpha x^\alpha$ is a circuit polynomial; see condition (4) of the definition of an ST-polynomial. The circuit number of the polynomial $g_\alpha$ is 
$\prod_{j\in {\rm nz(\alpha)}} (a_{\alpha,j}/\lambda_j^{(\alpha)})^{\lambda_j^{(\alpha)}}.$ As by hypothesis (ST2) $f_\alpha<0$ or 
$\alpha\not \in (2\NN)^n,$ condition (1) implies by Theorem \ref{thm:positiv} that $g_\alpha$ is nonnegative and 
by summing over all $\alpha \in \Delta(f)$ we obtain the {\sc sonc} 
\begin{eqnarray*}
 \sum_{\alpha\in \Delta(f)} g_\alpha & = & \sum_{j=0}^n \left(\sum_{\alpha \in \Delta(f)} a_{\alpha,j}\right) x^{\alpha(j)} + \sum_{\alpha\in \Delta(f)} f_\alpha x^\alpha.
\end{eqnarray*}

For each   $j=0,\ldots,n,$ the expression $(f_{\alpha(j)}-\sum_{\alpha\in \Delta(f)} a_{\alpha,j}) \cdot x^{\alpha(j)}$ is a monomial square and hence nonnegative by condition (2). Adding these expressions to one of the circuit polynomials, call it $g_{\tilde \alpha},$ 
we get a new circuit polynomial $\tilde g_{\tilde \alpha}  \geq  g_{\tilde \alpha}.$ Then  
$\tilde g_{\tilde \alpha} + \sum_{\alpha\in \Delta(f)\setminus \{\tilde \alpha\}} g_\alpha =f.$ Evidently, the Newton polytopes
of the circuit polynomials in this sum are all faces of the polytope $\New(f)$ and we are done with the first part.

Finally, if the $\alpha\in \Delta(f)$ are in $\New(f)^*$, then by Theorem \ref{thm:sos} the polynomials $g_\alpha$  are sums of binomial squares. Thus, by construction $\tilde g_{\tilde \alpha}$ is also a sum of  binomial squares. Therefore, $f$ is a sum of binomial squares. 
\end{proof} 

From this theorem we can deduce the following sufficient condition for the existence of a \textsc{sonc}-decomposition of an ST-polynomial. This condition depends on the coefficients of the polynomial alone.

\begin{cor}
 Let $f$ be an ST-polynomial in standard notation. Assume that  
 $f_{\alpha(j)} \geq \sum_{\alpha \in \Delta(f)} |f_\alpha| \lambda_j^{(\alpha)}$ for each $(\alpha,j)\in \Delta(f)\times \{0,\ldots,n\}.$
Then $f$ is a {\sc sonc}. If in addition $\Delta(f) \sbs \New(f)^*,$ then $f$ is {\sc sobs}.
\end{cor}

\begin{proof} 
Choose in Theorem \ref{thm:main1} $a_{\alpha,j}=|f_\alpha| \lambda_j^{(\alpha)}.$ Then the product in condition (1) of Theorem \ref{thm:main1} is $ |f_\alpha|$ and hence, that condition is satisfied. The corollary follows.
\end{proof}

This corollary yields a result in \cite{GP} as a consequence. 

\begin{cor}[\cite{GP}, Corollary 2.5]
 Let $f\in \RR[x]$ be a polynomial of degree $2d.$ If
 \begin{enumerate}
  \item $f_0 \geq \sum\limits_{\alpha\in \Delta} |f_\alpha|\frac{2d-|\alpha|}{2d}$ and
  \item $f_{2de_i} \geq \sum\limits_{\alpha\in \Delta} |f_\alpha|\frac{\alpha_i}{2d}$, for all $i = 1,\ldots,n$,
 \end{enumerate}
then $f$ is a sum of squares. 
\end{cor}

Since the proof of the corollary illustrates some points of the paper \cite{GP}, we give it here for the convenience of the reader.

\begin{proof}
 Write the polynomial in the form considered in \cite{GP} as $\tilde f =f_0+ \sum_{i=1}^n f_{2de_i} x^{2de_i} + \sum_{\alpha\in \Omega(f)} f_\alpha x^\alpha.$
Of course, this polynomial is a sum of squares if the truncated polynomial $f$ obtained by deleting the monomial square terms is a sum of squares. The truncated polynomial $f$ is ST. 
Then $f$ has a face of the simplex $\conv(V),$ $V=\{0, 2de_1,\ldots,2de_n\}$ as its Newton polytope.    
Since $\alpha= \sum_{i=1}^n\frac{ \alpha_i}{2d} (2de_i),$ we find $\lambda_0^{(\alpha)}=(2d-|\alpha|)/2d,$  and $\lambda_i^{(\alpha)}= \alpha_i/2d,$ for $i=1,\ldots,n$. Hence, the hypothesis of the 
present corollary translates into 
$f_{\alpha(j)} \geq \sum_{\alpha\in \Delta} |f_\alpha|\lambda_j^{(\alpha)},$ $j=0,1,\ldots,n.$ 
By Proposition \ref{prop:trellis} $\conv(V)$ 
is an $H$-simplex. So $\Delta(f) \sbs \New(f)^*.$ The corollary follows from the previous one.
\end{proof}

Theorem \ref{thm:main1} yields new sufficient criteria for a polynomial to be a {\sc sonc}  as well as to
be a sum of (binomial) squares. These criteria depend on the coefficients and the support
of the polynomial alone. They significantly extend previous sum of squares criteria given in
\cite{Fidalgo:Kovacec, GP, Lasserre} since we do not require the assumption that $\New(f)$ is a scaled standard simplex anymore.
All the polynomials treated in the cited literature are covered by the above theorems. \\

An important step to connect Theorem \ref{thm:main1} to geometric programming is given in the following 
theorem.

\begin{thm}
\label{thm:main2}
 Assume again that $f$ is an ST-polynomial in standard form and let $r\in \RR.$   Suppose
 that for every $(\alpha,j) \in \Delta(f) \times \{1,\ldots,n\}$ 
there exists an $a_{\alpha,j} \geq 0,$ such that:
\begin{enumerate}
\setcounter{enumi}{-1}
 \item If $\lambda_j^{(\alpha)}>0$, then  $a_{\alpha,j} > 0,$
 \item $|f_\alpha| \ \leq \ \prod\limits_{\substack{ j\in {\rm nz}(\alpha)\\ j\geq 1}} \left(\frac{a_{\alpha,j}}{\lambda_j^{(\alpha)}}\right)^{\lambda_j^{(\alpha)}}, 
     \text{ for every } \alpha \in \Delta(f) \text{ with } \lambda_0^{(\alpha)}=0,$
 \item $f_{\alpha(j)} \ \geq \ \sum\limits_{\alpha\in \Delta(f)} a_{\alpha,j} \text{ for all } j = 1,\ldots,n$
 \item $ f_0-r \ \geq \ \sum\limits_{\substack{\alpha \in \Delta(f) \\ 0 \in \rm nz(\alpha)}}  \lambda_0^{(\alpha)} |f_\alpha|^{1/\lambda_0^{(\alpha)}} 
     \prod\limits_{\substack{j\in {\rm nz}(\alpha) \\ j\geq 1}}  
         \left(\frac{\lambda_j^{(\alpha)}}{a_{\alpha,j}}\right)^{\lambda_j^{(\alpha)}/\lambda_0^{(\alpha)}}.$
\end{enumerate}
Then $f-r$ is a sum  of $|\Delta(f)|$ nonnegative circuit polynomials ({\sc sonc}) whose Newton polytopes are
 faces of $\conv(\{0\} \cup V(f))$.  
 \end{thm}

\begin{proof}
 First, note that the condition (0) is dispensable, provided that we assume the $a_{\alpha,j}$ to be chosen such that the other conditions are well-defined. In fact, by supposing this, condition (0) can be deduced from conditions (1) and (3).
Note that $f-r$ is an ST-polynomial again, since the right hand side of condition (3) is nonnegative. 
It is sufficient to show that the conditions given in the present theorem, which only considers $a_{\alpha,j}$
with $j\geq 1$, imply the existence of  
$a_{\alpha,j}$ as in Theorem \ref{thm:main1} (which also includes an $a_{\alpha,0}$), when the latter is formulated for $f-r$ in 
place of $f.$ As $f$ and $f-r$ differ  
in their constant term alone we observe that $\conv(\{0\} \cup V(f))=\conv(\{0\} \cup V(f-r))$ and $\Delta(f)=\Delta(f-r)$.

We fix an $\alpha\in \Delta(f-r)$ and investigate two cases:

Case $\lambda_0^{(\alpha)}=0$: Then $0\not\in {\rm nz}(\alpha)$. Hence, we have to consider condition (1) of Theorem \ref{thm:main2}. Since  $0\not\in {\rm nz}(\alpha)$ the additional assumption $j\geq 1$ under the product in condition (1) is obsolete.  This means that
the $\alpha$ under consideration satisfies condition (1) of Theorem \ref{thm:main1} independent of the choice of $a_{\alpha,0}$.

Case $\lambda_0^{(\alpha)}>0$: Then note that criterion (1) of Theorem \ref{thm:main1} is via solving for $a_{\alpha,0}$ equivalent to the inequality 
\begin{eqnarray}
 a_{\alpha,0} & \geq & \lambda_0^{(\alpha)} |f_\alpha|^{1/\lambda_0^{(\alpha)}} 
 \prod_{\substack{j\in {\rm nz}(\alpha) \\ j\geq 1}} 
  \left(\frac{\lambda_j^{(\alpha)}}{a_{\alpha,j}}\right)^{\lambda_j^{(\alpha)}/\lambda_0^{(\alpha)}}. \label{Equ:Proof1}
\end{eqnarray}
Once we have found $a_{\alpha,1},\ldots, a_{\alpha,n} \geq  0$ satisfying the conditions (2), (3) of the present theorem, 
there exists evidently the additional $a_{\alpha,0}>0$ to satisfy this inequality and hence (1) of Theorem \ref{thm:main1} for 
this $\alpha.$ 
The present conditions (2) coincide for all cases except $j=0$ with the conditions (2) in Theorem \ref{thm:main1}.
Condition (2) of Theorem \ref{thm:main1} for the current polynomial and $j=0$ requires us to find $a_{\alpha,0}$ such that
$f_0-r \geq \sum_{\alpha\in \Delta(f)} a_{\alpha,0}.$  
Now, as the $a_{\alpha,0}$ have to satisfy the inequality \eqref{Equ:Proof1} but are subject to no other conditions, we see that condition (3) guarantees what is required. 
\end{proof}

Finally, we prove the following theorem, which connects the conditions in the previous theorems to the convex cone of {\sc sonc}'s.

\begin{thm}\label{Thm:GPisSONC}
Let $f\in \RR[x]$ be an ST-polynomial in standard notation. We define:
\begin{itemize}
 \item $f_{\rm gp1}$ as the supremum of all $r\in \RR$  such that for every $\alpha \in \Delta(f)$ there exist nonnegative reals 
        $a_{\alpha,1},\ldots, a_{\alpha,n}$ such that the conditions (0) to (3) of Theorem \ref{thm:main2} are satisfied; and
 \item $f_{\rm gp2}$ as the supremum of all $r\in \RR$ such that there exist nonnegative circuit polynomials $g_1, g_2, \ldots, g_s$
whose Newton polytopes are faces of $\conv(\{0\} \cup V(f))$ such that $f-r =\sum_{k=1}^s g_k.$
\end{itemize}

Then these quantities are equal, i.e., $f_{\rm gp1}=f_{\rm gp2}.$
\end{thm}

\begin{proof}
Consider a real $r$ satisfying the conditions defining $f_{\rm gp1}.$ Then we know by Theorem \ref{thm:main2} that  there exist 
 $s=|\Delta(f)|$ nonnegative circuit polynomials whose  
Newton polytopes are faces of $\conv(\{0\} \cup V(f))$ such that
$f-r= \sum_{i=1}^s g_i.$ It follows that $f_{\rm gp2}\geq r$ and hence $f_{\rm gp2} \geq f_{\rm gp1}.$

Now consider a real $r$ satisfying the conditions defining $f_{\rm gp2}.$ Let $g_1,\ldots,g_s$ be circuit polynomials as
occurring in the equation. With $\alpha(g_k)$ denoting the tail monomial of $g_k,$ and obvious definitions of $g_{k,j},$  the equation reads in more detail
\begin{eqnarray*}
 f_0-r +\sum_{j=1}^n f_{\alpha(j)} x^{\alpha(j)} +\sum_{\alpha\in \Delta(f)} f_\alpha x^\alpha & = &
    \sum_{k=1}^s \underbrace{ ( g_{k,0}+ \sum_{j=1}^n g_{k,j}x^{\alpha(j)} +c_k x^{\alpha(g_k)})}_{=:g_k}.
\end{eqnarray*}

As $\Delta(f)$ and $V=\{0,\alpha(1),\ldots,\alpha(n)\}$ are disjoint and also $\{\alpha(g_1),\ldots,\alpha(g_s)\}$ and $V$
are disjoint, a comparison of coefficients of both sides of the equation implies that
\begin{eqnarray*}
 \{\alpha(g_i) \ : \ 1 \leq i \leq s \} & = & \Delta(f).
\end{eqnarray*}
We can assume that $s$ is minimal. Namely, if in this representation there would exist $l\neq l'$ such that
$\alpha(g_l)=\alpha(g_{l'}),$ then $g'=g_l+g_{l'}$ is a new nonnegative circuit polynomial with its Newton polytope being a face of $\conv(\{0\} \cup V(f)).$ We can use $g'$ to replace the subsum $g_l+g_{l'}$ above and obtain a representation with less than $s$ summands. This is a contradiction.

So, we henceforth assume that $s=|\Delta(f)|$ and that the circuit polynomials are indexed by $\alpha \in \Delta(f).$ If we set $g_l=g_\alpha$, then we have also $c_\alpha=f_\alpha$ and we get 
\begin{eqnarray*}
 f_0-r +\sum_{j=1}^n f_{\alpha(j)}x^{\alpha(j)}  +\sum_{\alpha\in \Delta(f)} f_\alpha x^\alpha & = &
    \sum_{\alpha\in \Delta(f)} \underbrace{ (g_{\alpha,0}+ 
         \sum_{j=1}^n g_{\alpha,j}x^{\alpha(j)} +f_\alpha x^{\alpha})}_{=:g_\alpha}.
\end{eqnarray*}
From this equality, the nonnegativity of circuit polynomials $g_\alpha,$ and Theorem \ref{thm:positiv} we get 
\begin{eqnarray}
 & & f_{\alpha(j)} \ = \ \sum_{\alpha\in \Delta(f)} g_{\alpha,j}, \quad  
f_0-r \ = \ \sum_{\alpha\in \Delta(f)} g_{\alpha,0}, \quad \text{ and } \quad 
|f_\alpha| \ \leq \ \prod_{j\in {\rm nz}(\alpha)} \left(\frac{ g_{\alpha,j} }{\lambda_j^{(\alpha)}}\right)^{\lambda_j^{(\alpha)}}. \label{Equ:Proof2}
\end{eqnarray}

From these equations and inequality, we arrange for each $\alpha \in \Delta(f)$ reals $a_{\alpha,j},$ with $j=1,\ldots,n,$  
satisfying the conditions of Theorem \ref{thm:main2}: we define for the cases $j=1,\ldots,n,$ $a_{\alpha,j}=g_{\alpha,j}.$ If $\lambda_j^{(\alpha)}>0$, then $j\in {\rm nz}(\alpha).$ Since $\alpha\in \Delta(f)$ we have $|f_\alpha|>0$. Thus, by the inequality in \eqref{Equ:Proof2}, also $g_{\alpha,j}>0.$ This shows that condition (0) is satisfied.
In the case $\lambda_0^{(\alpha)}=0,$ we have $0\not\in {\rm nz}(\alpha).$ In this case again the inequality guarantees (1).
We see that the condition (2) in Theorem \ref{thm:main2} is satisfied with equality. 

Finally, we investigate condition (3) of Theorem \ref{thm:main2}. For the case $0\in {\rm nz}(\alpha)$ we can solve the inequality for $g_{\alpha,0},$ obtaining that 
\begin{eqnarray*}
 g_{\alpha,0} & \geq & \lambda_0^{(\alpha)} |f_\alpha|^{1/\lambda_0^{(\alpha)}} \prod_{\substack{j\in {\rm nz}(\alpha) \\ j\geq 1}} \left(\frac{\lambda_j^{(\alpha)}}{a_{\alpha,j}}\right)^{\lambda_j^{(\alpha)}/\lambda_0^{(\alpha)}}.
\end{eqnarray*}
So, the equation for $f_0-r$ above implies condition (3) of Theorem \ref{thm:main2}.

Hence, a real $r$ satisfying the conditions defining $f_{\rm gp2}$  is also an $r$ such that there exist $a_{\alpha,1},\ldots, a_{\alpha,n}\geq 0$ satisfying the conditions of Theorem \ref{thm:main2}. Thus, $r \leq f_{\rm gp1},$ and consequently $f_{\rm gp2} \leq f_{\rm gp1}.$ This 
proves the theorem.
\end{proof}

In the following we make use of the unified notation $f_{\rm gp}$ corresponding to the equality $f_{\rm gp1}=f_{\rm gp2}.$
Ghasemi and Marshall observed a trade off between fast solvability of the corresponding geometric programs in comparison
to semidefinite programs and the fact that bounds obtained by geometric programs are
worse than $f_{\rm sos}$ for general polynomials $f$; \cite{GP}. Here, we conclude that this trade off does not
occur for polynomials with simplex Newton polytope satisfying the conditions of Theorem
\ref{thm:SOSMultiple}. Surprisingly, in this case the bound $f_{\rm gp}$ will be at least as good as the bound $f_{\rm sos}$.
Note that the special instance $\#\Delta(f) = 1$ and $\New(f)$ being the standard simplex with
edge length 2d was already observed by Ghasemi and Marshall; see \cite[Corollary 3.4]{GP}.

\begin{cor}\label{cor:gpbessersonc}
Let $f$ be an ST-polynomial in standard form and $\Delta(f) \subseteq \Int(\New(f)\cap \mathbb N^n)$. Suppose there exists $v\in (\R^*)^n$ such that
 $f_\alpha v^\alpha <0$ for all $\alpha\in \Delta(f).$
\begin{enumerate}
 \item Then $f_{\rm gp}=f^*.$
 \item If additionally $\Delta(f) \sbs \New(f)^*$, then $f_{\rm sos} = f^*.$
\end{enumerate}
\end{cor}

\begin{proof}
The polynomial $\tilde f=f-f^*$ is an ST-polynomial again which is nonnegative and has infimum $\tilde f^*=0.$ Clearly,
$\Delta(f)=\Delta(\tilde f)$. Thus, $\tilde f_\alpha v^{\alpha} = f_\alpha v^{\alpha} < 0$ for $\alpha\in \Delta(\tilde f).$ Hence, $\tilde f$ satisfies 
the hypotheses of Theorem \ref{thm:SOSMultiple}. Theorem \ref{thm:SOSMultiple} guarantees a {\sc sonc}-decomposition for $\tilde f$. Since  
$f - r$ is a polynomial that attains negative values for $r > f^*$ the polynomial $f - r$ cannot have a representation as a {\sc sonc}. 
In other words, from the definitions of $f_{\rm gp}$ and $f^*$ we get
\begin{eqnarray*}
	f_{\rm gp} & = & \sup\{r \ : \ f - r \text{ is a sonc }\} \ = \ f^*. 
\end{eqnarray*}
This shows the first statement.

For the second statement the proof is quite similar. Under the given additional hypothesis, we have 
$\Delta(f - f^*) \sbs \New(f - f^*)^*.$ Therefore, $f - f^*$ is a sum of binomial squares by Theorem \ref{thm:SBSMultiple}. If $r > f^*$, then $f - r$ assumes negative values and thus cannot be sum of squares. 
Hence, from the definitions of $f_{\rm sos}$ and $f^*$ we get
\begin{eqnarray*}
f_{\rm sos} & = & \sup\{r \ : \ f - r \mbox{ is a sos }\} \ = \ f^*.
\end{eqnarray*}
\end{proof}

\section{Geometric Programming}
\label{Sec:GeometricProgramming}
In this section, we prove that the number $f_{\rm gp}$ can indeed be obtained by a geometric program, which we introduce first.

\begin{definition}
A function $p : \mathbb R_{>0}^n\to \mathbb R$ of the form $p(z) = p(z_1,\ldots,z_n) = cz_1^{\alpha_1}\cdots z_n^{\alpha_n}$ with $c > 0$ and $\alpha_i \in \mathbb R$ is called a monomial (function). A sum $\sum_{i=0}^k c_iz_1^{\alpha_{1}(i)}\cdots z_n^{\alpha_{n}(i)}$ of monomials with $c_i > 0$ is called a posynomial (function).
\label{Def:GP}
\end{definition}

A \emph{geometric program} has the following form.
\begin{eqnarray}
\label{Equ:GP}
\begin{cases} 
\text{minimize} & p_0(z), \\ 
\text{subject to:} &  
\begin{array}{cl}
 (1) & p_i(z) \leq 1 \ \text{ for all } \ 1\leq i\leq m, \\
 (2) & q_j(z) = 1 \ \text{ for all } \ 1\leq j\leq r, \\
\end{array}
\end{cases}
\end{eqnarray}
where $p_0,\dots,p_m$ are posynomials and $q_1,\dots,q_r$ are monomial functions.

Geometric programs can be solved with interior point methods. In \cite{nesterov}, the authors prove  worst-case polynomial time complexity of this method. For an introduction and practical abilities of geometric programs see \cite{Boyd:GP, Boyd:CO}. Based on our main results, Theorems \ref{thm:main1} and \ref{thm:main2}, we can conclude the following corollary.

\begin{cor}\label{cor:gp}
Let $f\in \R[x]$ be an ST-polynomial of degree $2d$ in standard form. Let $R$ be the subset of an $n |\Delta(f)|$-dimensional real space given by
\begin{eqnarray*}
 R & = & \{(a_{\alpha,i}) \ : \ a_{\alpha,i} \in \R_{> 0} \text{ for every } (\alpha,i)\in \Delta(f)\times \{1,\ldots,n\}\}.
\end{eqnarray*}
Then $f_{\rm gp}=f_0-m^*,$ where $m^*$ is given as the output of the following geometric program:
\begin{eqnarray*}
\begin{cases}
\emph{minimize} & \sum\limits_{\substack{\alpha\in \Delta(f) \\ \lambda_0^{(\alpha)} \neq 0}}\lambda_0^{(\alpha)} |f_\alpha|^{1/\lambda_0^{(\alpha)}} 
    \prod\limits_{\substack{ j\in {\rm nz}(\alpha)\\ j\geq 1}}
  \left(\frac{\lambda_j^{(\alpha)}}{a_{\alpha,j}}\right)^{\lambda_j^{(\alpha)}/\lambda_0^{(\alpha)}}
\emph{ over the subset } R' \emph{ of } R \\
& \\
\emph{defined by} &
\begin{array}{cl}
 (1) & \sum\limits_{\alpha\in \Delta(f)} (a_{\alpha,j}/f_{\alpha(j)} ) \leq 1 \emph{ for every } 1 \leq j \leq n. \\
 (2) & |f_\alpha| \prod\limits_{\substack{j\in {\rm nz}(\alpha)\\ j\geq 1}}
                    \left(\frac{\lambda_j^{(\alpha)}}{a_{\alpha,j}}\right)^{ \lambda_j^{(\alpha)}} \leq 1 \emph{ for every } \alpha\in \Delta(f) \emph{ with } \lambda_0^{(\alpha)}=0. \\
\end{array}
\end{cases}
\end{eqnarray*}
\end{cor}
 
\begin{proof}
Let $M((a_\alpha))$ denote the function to minimize. From the definitions it follows that the above qualifies as a geometric program, since $M((a_\alpha))$ is a posynomial in the variables $a_{\alpha,j}$ for $(\alpha,j)\in \Delta(f)\times \{1,\ldots,n\}$. Moreover, the functions entering in the constraint inequalities are posynomials  in the same variables; in fact the functions in constraint (1) are even linear.
By Theorem \ref{Thm:GPisSONC}, $f_{\rm gp}=\sup\{r \ : \ f_0-r \text{ satisfies the conditions of Theorem \ref{thm:main2}}\}.$  We see that 
\begin{eqnarray*}
 f_{\rm gp} & = & \sup\{r \ : \ f_0-r=M((a_\alpha)), \text{ and the } (a_{\alpha}) \text{ satisfy the inequalities defining } R'\}  \\
 & = & \sup\{f_0-M((a_\alpha)) \ : \ \text{the } (a_{\alpha}) \text{ satisfy the inequalities defining } R'\} \\
 & = & f_0-\inf\{M((a_\alpha)) \ : \ \text{the } (a_{\alpha}) \text{ satisfy the inequalities defining } R'\},
\end{eqnarray*}
which completes the proof. 
\end{proof}

\subsection{Examples}
\label{SubSec:Examples}

We demonstrate our method and reflect our results by five examples. All following geometric programs are solved via the \textsc{Matlab} solver \textsc{gpposy}\footnote[1]{The
Matlab version used was R2011a, running on a desktop computer with Intel(R) Core(TM)2 @ 2.33 GHz and 2 GB of RAM.}.
\begin{enumerate}
 \item  First, consider the polynomial $f = \frac{1}{4} + x_1^8 + x_1^2x_2^6 + 4x_1^3x_2^3$.
The geometric program proposed in \cite{GP} is infeasible, since the pure power $x_2^8$ is missing in the polynomial to make the Newton polytope a standard simplex of edge length $8$. However, $\New(f)$ is an $H$-simplex and we can use our results to compute $f_{\rm gp}$. Here, we have $\Delta = \{\alpha\} = \{(3,3)\}$. So, we introduce the variables $a_{\alpha,j}$ for $j\in\{1,2\}$. Therefore, by Corollary \ref{cor:gp}, we have to solve the following geometric program:
\begin{eqnarray*}
& & \inf\left\{\frac{1}{4} \cdot 4^4 \cdot \left(\frac{1}{4}\right)^{\frac{4}{4}} \cdot \left(\frac{1}{2}\right)^{\frac{4}{2}} \cdot a_{\alpha,1}^{-1} a_{\alpha,2}^{-2} \ : \ a_{\alpha,1}\leq 1, a_{\alpha,2}\leq 1\right\}.
\end{eqnarray*}
The optimal solution is given by $a_{\alpha,1} = a_{\alpha,2} = 1$ yielding $m^* = 4$ and hence $f_{\rm gp} = \frac{1}{4} - 4  = -3.75 = f_{\rm sos} = f^*$ by Corollary \ref{cor:gpbessersonc}.\\
\item Let $f = \frac{187}{208} + x_1^{80}+x_2^{78}-8x_1^5x_2^3$. Again, the geometric program proposed in \cite{GP} is infeasible. However, $\New(f)$ is an $H$-simplex and with $\lambda^{(5,3)}_1 = 1/16$ and $\lambda^{(5,3)}_2 = 1/26$ for the convex combination of the interior exponent. Thus, our corresponding geometric program is given by
\begin{eqnarray*}
& & \inf\left\{\frac{187}{208}\cdot 8^{\frac{208}{187}}\cdot \left(\frac{1}{16}\right)^{\frac{13}{187}}\cdot \left(\frac{1}{26}\right)^{\frac{8}{187}}\cdot a_{\alpha,1}^{-\frac{13}{187}}\cdot a_{\alpha,2}^{-\frac{8}{187}} \ : \ a_{\alpha,1}\leq 1, a_{\alpha,2}\leq 1\right\}.	
\end{eqnarray*}
Using the software \textsc{Gloptipoly}, see \cite{gloptipoly}, $f^* \approx -5.6179$ was computed in $4327.2$ seconds, i.e., approximately $\mathbf{1.2}$ \textbf{hours}. In contrast, using the geometric program of Corollary \ref{cor:gp}, we get a global minimizer $a_{\alpha,1} = a_{\alpha,2} = 1$ and the optimal solution $m^* = \frac{187}{208}\cdot \left(\frac{8^{208}}{16^{13} \cdot 26^8}\right)^\frac{1}{187}$ and, hence, $f^* = \lambda_0 - m^* = \frac{187}{208} \cdot \left(1 - \left(\frac{8^{208}}{16^{13} \cdot 26^8}\right)^\frac{1}{187}\right) \approx -5.6179$ in $\mathbf{0.5}$ \textbf{seconds}.\\
\item Let now $f = \frac{17}{20} + 3x_1^8x_2^4 + 2x_1^6x_2^8 - 10x_1^3x_2^3 + x_1^5x_2^4$. Again, the geometric program in \cite{GP} cannot be used but the geometric program in Corollary $\ref{cor:gp}$ with $\Delta = \{\overline{\alpha},\alpha\} = \{(3,3),(5,4)\}$ now reads as follows.
\begin{eqnarray*}
& & \inf \left\{\frac{9}{1250}\cdot 2^{\frac{1}{3}}\cdot 5^{\frac{2}{3}}\cdot a_{\alpha,1}^{-\frac{4}{3}}\cdot a_{\alpha,2}^{-1} + \frac{11}{40}\cdot 10^{\frac{3}{11}}\cdot 3^{\frac{9}{11}}\cdot 20^{\frac{8}{11}}({a_{\overline{\alpha},1}})^{-\frac{3}{11}}({a_{\overline{\alpha},2}})^{-\frac{6}{11}}\right\} \\
& & \text{such that } \ \frac{a_{\alpha,1}+a_{\overline{\alpha},1}}{3} \, \leq \, 1 \ \text{ and } \ \frac{a_{\alpha,2}+a_{\overline{\alpha},2}}{2} \, \leq \, 1.
\end{eqnarray*}
Here, the variables $a_{\alpha,j}$ come from $\alpha = (5,4)$ and $a_{\overline{\alpha},j}$ come from $\overline{\alpha} = (3,3)$.
Again, we use the \textsc{Matlab} solver \textsc{gpposy} to solve this geometric program with the following code:

\begin{verbatim}
>> A0=[-4/3,-1,0,0;0,0,-3/11,-6/11]
>> A1=[1,0,0,0;0,0,1,0]
>> A2=[0,1,0,0;0,0,0,1]
>> A=[A0;A1;A2]
>> b0=[9/1250*2^(1/3)*5^(2/3);11/40*10^(3/11)*3^(9/11)*20^(8/11)]
>> b1=[1/3;1/3]
>> b2=[1/2;1/2]
>> b=[b0;b1;b2]
>> szs=[size(A0,1);size(A1,1);size(A2,1)]
>> [x,status,lambda,nu]=gpposy(A,b,szs)
\end{verbatim}

The optimal solution is given by $$(a_{\alpha,1},a_{\alpha,2},a_{\overline{\alpha},1},a_{\overline{\alpha},2}) \ = \ (0.5910,0.1685,2.4090, 1.8315)$$
yielding
$m^* \approx -6.644$ and hence $f_{\rm gp} = \frac{17}{20} - 6.644 \approx - 5.794$. We also have $f_{\rm sos} =  f_{\rm gp} = f^*$ by Corollary \ref{cor:gpbessersonc}.\\

\item The Motzkin polynomial $f = \frac{1}{3} + \frac{1}{3}x_1^4x_2^2 + \frac{1}{3}x_1^2x_2^4 - x_1^2x_2^2$ satisfies $f_{\rm gp} = f^* = 0$ again by Corollary \ref{cor:gpbessersonc}. However, $f_{\rm sos} = -\infty$.\\

\item Let $f=\frac{5}{12} + \frac{5}{24}x_1^6 + \frac{5}{24}x_1^2x_2^4 + \frac{5}{24}x_1^2x_2^2 - \frac{5}{8}x_1x_2$. Note that this $f$ is not even ST in view of the square $\frac{5}{24}x_1^2x_2^4$ and $(2,2)^T$ being contained
in the interior of its Newton polytope. However, one can still define a GP in the sense of Corollary \ref{cor:gp}. In this case one can check that
$$f_{\rm gp} \ \approx \ -0.41 \ < \ f_{\rm sos} \ = \ f^* \ \approx \ 0.196.$$
\end{enumerate}

\section{Applications to Constrained Polynomial Optimization}
\label{Sec:Constrained}

Based on our previous results we can derive some initial applications for constrained polynomial optimization. We follow the set-up in \cite{GP:semialgebraic}, which we recall here first. Let $g_1,\dots, g_s \in \R[x] = \R[x_1, \dots, x_n]$ and consider the semialgebraic set
$$K \ = \ \{x \in \mathbb R^n : g_i(x) \geq 0,\,\, 1\leq i\leq s\}.$$
We consider the constrained polynomial optimization problem 
$$f_K^* \ = \ \inf_{x \in K}^{}f(x).$$
For a polynomial $f \in \R[x]$, in order to derive lower bounds for $f_K^*$, we define a new function given by
$$h(\mu) \ = \ f - \sum_{j=1}^{s} \mu_jg_j$$
for $\mu = (\mu_1,\dots,\mu_s) \in  [0,\infty)^s$. Note that for every fixed $\mu$ the function $h(\mu)$ is a polynomial in $\R[x]$. In \cite{GP:semialgebraic} Ghasemi and Marshall use a method first developed in their paper \cite{GP} to obtain a lower bound for $h(\mu)$ via geometric programming for every given $\mu$. We denote this bound as $h(\mu)_{GM}$ in order to distinguish it from our bound given by geometric programming, which we introduced in the previous sections. Let  $\mathbf{g} = (g_1,\dots ,g_s)$. Since there exists a bound $h(\mu)_{GM}$ for every choice $\mu \in  [0,\infty)^s$, we can take the supremum $s(f,\mathbf{g})_{GM}$ and obtain
$$s(f,\mathbf{g})_{GM} \ = \ \sup\{h({\mu})_{GM} : \mu \in  [0,\infty)^s\} \ \leq \ f_K^*.$$
Ghasemi and Marshall's approach does not allow in general to compute $s(f,\mathbf{g})_{GM}$ via a geometric program directly, but they give an optimization problem \cite[program (2) on page 4]{GP:semialgebraic}, which can be relaxed to a geometric program; see \cite[Theorem 4.1. and 4.2. (1)]{GP:semialgebraic}. Here, we show that a similar optimization problem can be formulated in our setting.\\

Assume that for a given $\mu \in [0,\infty)^s$ the Newton polytope $\New(h(\mu))$ is a simplex with even vertex set $\{0, \alpha(1),\dots,\alpha(n)\} \subset (2\N)^n$. The results in Sections \ref{Sec:MainResults} and \ref{Sec:GeometricProgramming} imply that $h(\mu)_{gp}$ is a lower bound for $h(\mu)$ on $\mathbb R^n$, which also implies that it is a lower bound for $f$ on the semialgebraic set $K$. Namely, let, again,   $\mathbf{g} = (g_1,\dots ,g_s)$, and let $h(\mu)_{gp}$ be the optimal value of the geometric program introduced in Section \ref{Sec:GeometricProgramming} for the polynomial $h(\mu)$. Then we obtain similarly as in \cite{GP:semialgebraic}
$$s(f,\mathbf{g}) \ = \ \sup\{h(\mu)_{gp} : \mu \in [0,\infty)^s\} \ \leq \ f_K^*.$$
Conveniently, write $h(\mu) = -\sum_{j=0}^s \mu_jg_j$ for $g_0 = -f$ and $\mu_0 = 1$ and define $\Delta(h(\mu))$ in the usual sense as the set of exponents of $h(\mu)$ not defining a monomial square. Moreover, we define $\Delta(h) = \Delta(f)\cup\Delta(-g_1)\cup\dots\cup\Delta(-g_s)$. Note that $\Delta(h(\mu)) \subseteq \Delta(h)$ for all $\mu$.

We remark that for a given $\mu$ the geometric program from Section \ref{Sec:GeometricProgramming} is feasible only if $h(\mu)$ is an ST-polynomial and hence particularly only if $\New(h(\mu))$ is a simplex. If this is not the case, then we set $h(\mu)_{gp} = -\infty$. Note in this context that the support as well as the Newton polytope of $h(\mu)$ can change if certain $\mu_j$ equal $0$ or if term cancellation occurs. Additionally, we can assume that all $-g_j$ do not contain monomial squares; see \cite[Section 3, second paragraph]{GP:semialgebraic}. Analogously as in the previous sections, we denote by $\{\lambda_0^{(\alpha,\mu)},\dots,\lambda_n^{(\alpha,\mu)}\}$ the barycentric coordinates of the lattice point $\alpha \in (\New(h(\mu))\cap \mathbb N^n)$ with respect to the vertices of the simplex $\New(h(\mu))$.
For every $\alp \in \Delta(h)$ we define a set
\begin{eqnarray*}
 R_{\alp} & = & \{\mathbf{a}_{\alp} \ : \ \mathbf{a}_{\alp} = (a_{\alp,1},\ldots,a_{\alp,n}) \in \R_{>0}^n\}.
\end{eqnarray*}
Furthermore, we define the nonnegative  real set $R$ as
\begin{eqnarray*}
R & = & [0,\infty)^s \times \bigtimes_{\alp \in \Delta(h)} (R_\alp \times \R_{\geq 0}).
\end{eqnarray*}
Hence, $R$ is the Cartesian product of $[0,\infty)^s$ and $|\Delta(h)|$ many copies $\R_{>0}^n \times \R_{\geq 0}$; each given by one $R_{\alp}$ with $\alp \in \Delta(h)$ and a $\R_{\geq 0}$. We define the function $p$ from $R$ to $\R_{\geq 0}$ as
\begin{eqnarray*}
 & & p(\mu,\{(\mathbf{a}_\alp,b_\alp) \ : \ \alp \in \Delta(h)\}) \ = \ \\
 & & \sum_{j=1}^{s} \mu_jg_j(0) + 
 \sum_{\substack{\alpha \in \Delta(h) \\ \lambda_0^{(\alpha,\mu)} \neq 0}} \lambda_0^{(\alpha,\mu)} \cdot b_{\alpha}^{\frac{1}{\lambda_0^{(\alpha,\mu)}}}\cdot \prod_{\substack{j\in {\rm nz}(\alpha) \\ j\geq 1}} \left(\frac{\lambda_j^{(\alpha,\mu)}}{a_{\alpha,j}}\right)^{\frac{\lambda_j^{(\alpha,\mu)}}{\lambda_0^{(\alpha,\mu)}}}
\end{eqnarray*}
where, analogously as before, $h(\mu)_{\alpha}$ denotes the coefficient of the term with exponent $\alpha$ of $h(\mu)$. \\

We consider the following optimization problem:

\begin{eqnarray}
& & \begin{cases}
\text{minimize} &
p(\mu,\{(\mathbf{a}_\alp,b_\alp) \ : \ \alp \in \Delta(h)\})
\text { over the subset of } R
\\
& \\
\text{defined by:} & 
\begin{array}{cl}
 (1) & \sum\limits_{\alpha\in\Delta(h)} a_{\alpha,j} \, \leq \, h(\mu)_{\alpha(j)} \ \text{ for all } \ 1 \leq j \leq n \ \text{ and } \\
 (2) & \prod\limits_{\substack{j\in {\rm nz}(\alpha)\\ j\geq 1}} \left(\frac{a_{\alpha,j}}{\lambda_j^{(\alpha,\mu)}}\right)^{\lambda_j^{(\alpha,\mu)}} \, \geq \,  |h(\mu)_{\alpha}| 
\begin{array}{l}
 \ \text{for every } \alpha\in \Delta(h)\\
 \ \text{with } \lambda_0^{(\alpha,\mu)}=0 \\
\end{array} \\
(3) & |h(\mu)_{\alpha}| \leq b_\alp
 \begin{array}{l}
 \ \text{for every } \alpha\in \Delta(h)\\
 \ \text{with } \lambda_0^{(\alpha,\mu)}\neq0 \\
\end{array} \\
\end{array}
\end{cases}
\label{Equ:Constrained}
\end{eqnarray}
Note that for a given $\mu \in [0,\infty)^s$ constraint (3) allows to choose $b_\alp = 0$ if $\alp \in \Delta(h)\setminus\Delta(h(\mu))$. We remark that constraint (3) is only necessary in order to be able to transform this program into a geometric program under certain conditions; see Theorem \ref{Thm:GPConstrainedCaseSignomial}. Namely, without the constraint (3) we would have to replace the term $b_{\alpha}^{\frac{1}{\lambda_0^{(\alpha,\mu)}}}$ by $|h(\mu)_{\alpha}|^{\frac{1}{\lambda_0^{(\alpha,\mu)}}}$ in $p$. But $|h(\mu)_{\alpha}|^{\frac{1}{\lambda_0^{(\alpha,\mu)}}}$ is not a posynomial in general.

A key observation is that the optimal value of this optimization problem yields a lower bound for $s(f,\mathbf{g})$. This is an analogue result to \cite[Theorem 3.1.]{GP:semialgebraic}.

\begin{thm}\label{Thm:GPConstrainedCase}
Let $\gamma$ be optimal value of the above optimization problem. Then $f_0 - \gamma \leq s(f,\mathbf{g})$.
\end{thm}

\begin{proof} 
Consider a fixed $\mu$ such that the optimization problem \eqref{Equ:Constrained} restricted to this $\mu$ is feasible. Choose $b_\alp = |h(\mu)_\alp|$ and apply Corollary \ref{cor:gp} to the polynomial in $x \in \R^n$, $h(\mu) = h(\mu)(x)$. Then 
$$\sum_{\substack{\alpha \in \Delta(h) \\ \lambda_0^{(\alpha,\mu)} \neq 0}}^{}\lambda_0^{(\alpha,\mu)} \cdot b_\alp^{\frac{1}{\lambda_0^{(\alpha,\mu)}}}\cdot \prod_{\substack{j\in {\rm nz}(\alpha)\\ j\geq 1}} \left(\frac{\lambda_j^{(\alpha,\mu)}}{a_{\alpha,j}}\right)^{\frac{\lambda_j^{(\alpha,\mu)}}{\lambda_0^{(\alpha,\mu)}}}$$
is an upper bound for $\gamma(h(\mu)) = h(\mu)(0) - h(\mu)_{gp}$. Note that $|b_\alp| = 0$ for all $\alp \in \Delta(h) \setminus \Delta(h(\mu))$. Hence, for a fixed $\mu$ every feasible point $(\mu,\{(\mathbf{a}_\alp,b_\alp) \ : \ \alp \in \Delta(h)\})$ of
\eqref{Equ:Constrained}
$$f_0 - \sum_{j=1}^{s} \mu_jg_j(0) - \sum_{\substack{\alpha \in \Delta(h) \\ \lambda_0^{(\alpha,\mu)} \neq 0}}^{}\lambda_0^{(\alpha,\mu)} \cdot b_\alp^{\frac{1}{\lambda_0^{(\alpha,\mu)}}}\cdot \prod_{\substack{j\in {\rm nz}(\alpha)\\ j\geq 1}} \left(\frac{\lambda_j^{(\alpha,\mu)}}{a_{\alpha,j}}\right)^{\frac{\lambda_j^{(\alpha,\mu)}}{\lambda_0^{(\alpha,\mu)}}}$$
is a lower bound for $h(\mu)_{gp} = h(\mu)(0) - \gamma(h(\mu))$. It follows that $f_0 - \gamma$ is a lower bound for $s(f,\mathbf{g})$.

\end{proof}

Although $h(\mu)_{gp}$ is a geometric program, this is not true in general for the optimization problem \eqref{Equ:Constrained}. Actually, \eqref{Equ:Constrained} is not even a \textit{signomial program}, i.e., a program defined as a geometric program except that the nonnegativity condition of the coefficients dropped. This is due to the absolute sign $|h(\mu)_\alp|$ in condition (3) depending on the choice of $\mu$. Only if additional assumptions are satisfied one can guarantee that \eqref{Equ:Constrained} is a signomial or geometric program, as we show in the following theorem. However, in our follow-up paper \cite{Dressler:Iliman:deWolff:GPConstrained} we examine conditions on how \eqref{Equ:Constrained} can be turned into signomial programs and subsequently relaxed to geometric programs in order to find lower bounds for $f_K^*$.

\begin{thm}
The optimization problem \eqref{Equ:Constrained} restricted to $\mu \in (0,\infty)^s$ is a signomial program if for every $\alpha$ it holds that $h(\mu)_\alpha$ has the same sign for every choice of $\mu$. If additionally for every $j = 1,\ldots,n$ it holds that $h(\mu)_{\alp(j)}$ is a monomial, all coefficients $h(\mu)_{\alpha(j)}$ corresponding to the vertices $\{\alpha(1),\ldots,\alpha(n)\}$ of $\New(h(\mu))$ are strictly positive, all $h(\mu)_\alpha$ are strictly positive, and all $g_j(0)$ for $j = 1,\ldots,s$ are greater or equal than zero, then \eqref{Equ:Constrained} is a geometric program.
\label{Thm:GPConstrainedCaseSignomial}
\end{thm}

\begin{proof}
If we restrict \eqref{Equ:Constrained} to $\mu \in (0,\infty)^s$, then all involved functions are almost signomials by definition, since all the variables $\mu_1,\ldots,\mu_s$ and $a_{\alpha,1},\ldots,a_{\alpha,n}$ for all $\alpha \in \Delta(h)$ are real and strictly positive. The only remaining issue is the absolute value of $|h(\mu)_\alpha|$ in the constraints (2) and (3). Since $h(\mu)_\alpha = \mu_0 (g_0)_\alpha +
\mu_1 (g_1)_\alpha + \cdots + \mu_s (g_s)_\alpha$, the sign of $h(\mu)_\alpha$ in general depends on $\mu$ such that we cannot express the absolute value unless $\mu$ is fixed. But if the sign of $h(\mu)_\alpha$ is constant for every choice of $\mu$, then we can write $|h(\mu)_\alpha|$ as $\pm h(\mu)_\alpha$, that is $\pm(\mu_0 (g_0)_\alpha +
\mu_1 (g_1)_\alpha + \cdots + \mu_s (g_s)_\alpha)$ in the constraints (2) and (3). This turns the constraints (2) and (3) into signomials or even posynomials if additionally all $h(\mu)_\alpha$ have positive coefficients. Hence, the first part follows with the definition of a signomial and Definition \ref{Def:GP}. In order to obtain an expression of the form that a posynomial is $\leq 1$ in constraint (1) the $h(\mu)_\alp(j)$ needs to be a \textit{monomial}, which the left hand side can be divided by. While this is not true in general, this condition is guaranteed by the assumptions of the theorem. Moreover, the sums $h(\mu)_{\alpha(j)}$ for $j = 1,\ldots,n$ and $\sum_{j = 1}^s \mu_j g_j(0)$ may not involve a negative sign in order to be a posynomial. This also is satisfied by assumption. All remaining terms were already considered in the first part and hence the second part follows.
\end{proof}

We close the section with an example showing the capabilities of our program compared to the corresponding one by Ghasemi and Marshall in \cite{GP:semialgebraic}.

\begin{example}
Let $m(x,y) = 1 + x^4y^2 + x^2y^4 - 3x^2y^2$ be the Motzkin polynomial. Assume that we would like to minimize the Motzkin polynomial on the semi-algebraic set
\begin{eqnarray*}
 K & = & \{(x,y) \in \R^2 \ : \ x \cdot y \geq 0\}.
\end{eqnarray*}
In other words, $K = \R^2_{\geq 0} \cup \R^2_{\leq 0}$. By construction it follows that $h(\mu)(x,y) = m(x,y) - \mu \cdot x y$. Since $m(x,y)$ has global minimizers $(1,1)$ and $(-1,-1)$ we have $m^* = m^*_{|K}$. For $\mu = 0$ we have $h(0)(x,y) = m(x,y)$ and thus $h(0)_{gp} = m_{\rm gp}$. By example (4) in Section \ref{SubSec:Examples} we know $m_{\rm gp} = m^* = 0$. Moreover, if we restrict the optimization program \eqref{Equ:Constrained} applied on $h(\mu)$ to the case $\mu = 0$, then it coincides with our geometric program from Corollary \ref{cor:gp} applied on $m(x,y)$. This implies $m_{\rm gp} \leq 1 - \gamma$ where $\gamma$ denotes the optimal value of the optimization program \eqref{Equ:Constrained}. Hence, we obtain in total
\begin{eqnarray*}
 & & 0 \ \leq \ m_{\rm gp} \ = \ h(0)_{gp} \ \leq \ 1 - \gamma \ \leq \ s(m,xy) \ \leq \  m^*_{|K} \ = \ m^* \ = \ m_{\rm gp} \ = \ 0.
\end{eqnarray*}
Thus, we can conclude
\begin{eqnarray*}
 & & 1 - \gamma \ = \ s(m,xy) \ = \ m^*_{|K} \ = \ m^* \ = \ 0.
\end{eqnarray*}

Now, we try to apply the program proposed by Ghasemi and Marshall \cite[Program (2) before Theorem 3.1.]{GP:semialgebraic} on the same problem. Since the term $xy$ corresponds to an interior point in the Newton polytope of $\New(m)$, we have $\New(h(\mu)) = \New(m)$ for every choice of $\mu$. But $\New(m)$ is not a scaled standard simplex. Hence, the program by Ghasemi and Marshall is infeasible for all choices of $\mu$ and therefore it is infeasible in total.
\end{example}

\section{Conclusion and Outlook}
We have proposed a new geometric program providing lower bounds for polynomials that significantly extends the existing one in \cite{GP}. This extension sheds light to the crucial structure of the Newton polytope of polynomials. In particular, our results serve as a next step in optimization of polynomials with simplex Newton polytopes and connect this problem to
\begin{enumerate}
	\item the recently established {\sc sonc} nonnegativity certificates, and
	\item the construction of simplices with an interesting lattice point structure, namely, what we have called $H$-simplices in this article.
\end{enumerate}
In \cite{GP} Ghasemi and Marshall observed a trade off between bounds based on geometric programming $f_{\rm gp}$ and bounds based on semidefinite programming $f_{\rm sos}$. The bound $f_{\rm sos}$ was better but its computation took far longer. This trade off cannot be observed in our refinement. While fast solvability of the geometric programs still seems to hold, the bounds $f_{\rm gp}$ and $f_{\rm sos}$ are not comparable in general. There exist classes for which $f_{\rm gp}\leq f_{\rm sos}$ holds \cite{GP}, but there also exist classes with $f_{\rm gp}\geq f_{\rm sos}$ (Corollary \ref{cor:gpbessersonc}). The latter is a crucial case due to the fast solvability of geometric programs. It would be interesting to discover more classes for which the bounds are comparable. Hence, an analysis of the gap $f_{\rm sos} - f_{\rm gp}$ is an important task having major impact on computational complexity of solving polynomial optimization problems. Equivalently, looking from a convex geometric viewpoint, it would be interesting to analyze the gap between the cone of sums of squares and the cone of sums of nonnegative circuit polynomials as well as the gap between the cone of nonnegative polynomials and the cone of sums of nonnegative circuit polynomials.

Furthermore, we showed that the methods developed in this paper cannot only be applied to global but also to constrained polynomial optimization problems. In the follow-up paper \cite{Dressler:Iliman:deWolff:GPConstrained} we discuss the constrained case and particularly the optimization problem \eqref{Equ:Constrained} in more detail, thereby extending results in Section \ref{Sec:Constrained} and results in \cite{GP:semialgebraic}. We will also connect these problems again to the geometry of the convex cone of {\sc sonc}'s.

\section*{Acknowledgments}

\noindent We are deeply grateful to Alexander Kovacec for his detailed comments and suggestions, which significantly improved the presentation of the paper. We also thank Mareike Dressler and Kaitlyn Philipson for their helpful comments and the anonymous referees for their various suggestions and helpful comments.\\

\noindent The second author was partially supported by DFG grant MA 4797/3-2.

\bibliographystyle{amsalpha}
\bibliography{gp}

\end{document}